\documentclass[10pt,letterpaper,oneside]{amsart}

  \usepackage{amsmath,amssymb,amsthm}
  \usepackage{amscd}
  \usepackage[margin=1.3in,footskip=1.5in]{geometry}
  \usepackage{graphicx,epstopdf}
  \usepackage{color}
  \usepackage{enumerate}
  \usepackage{xspace}
  \usepackage{tipa}
\usepackage{latexsym}
  \usepackage{epsfig}
  \usepackage{pinlabel}
  \usepackage[all]{xy}

\usepackage{xcolor}
\definecolor{blue}{rgb}{0,0.5,0.5}
\definecolor{ocean}{rgb}{0.00,0.26,0.50}
\usepackage[colorlinks = true, linkcolor=blue, citecolor=ocean]{hyperref}
\usepackage{enumitem}
 \usepackage{mathabx}
\newcommand{\showcommentsbox}{yes}

\newsavebox{\commentbox}
{\ifthenelse{\equal{\showcommentsbox}{yes}}%
{\footnotemark
        \begin{lrbox}{\commentbox}
        \begin{minipage}[t]{1.25in}\raggedright\sffamily\tiny
        \footnotemark[\arabic{footnote}]}
{\begin{lrbox}{\commentbox}}}%
{\ifthenelse{\equal{\showcommentsbox}{yes}}%
{\end{minipage}\end{lrbox}\marginpar{\usebox{\commentbox}}}
{\end{lrbox}}}

  \newcommand{\calC}{\mathcal{C}}

  \newcommand{\calF}{\mathcal{F}}
  
  \newcommand{\calH}{\mathcal{H}}

  \newcommand{\calX}{\mathcal{X}}
  \newcommand{\calY}{\mathcal{Y}}

  \newcommand{\EE}{\mathbb{E}}

  \newcommand{\ZZ}{\mathbb{Z}}

  \newcommand{\gothic}{\mathfrak}

  \newcommand{\gS}{{\gothic S}}

  \newtheorem{theorem}{Theorem}[section]
  \newtheorem{proof of the main theorem}[theorem]{Proof of the Main Theorem}
  \newtheorem{proposition}[theorem]{Proposition}
   \newtheorem*{cormcg}{Corollary~\ref{cor: mcg}}
  \newtheorem*{corfbc}{Corollary~\ref{cor:non-existence-ct-free-by-cyclic}}
  \newtheorem*{mainthm1}{\textbf{Theorem}~\ref{main-thm-1}}
  \newtheorem*{mainthm2}{\textbf{Theorem}~\ref{main-thm-2}}
  \newtheorem{sketch of proof}[theorem]{Sketch of Proof}  
  \newtheorem{corollary}[theorem]{Corollary}
  \newtheorem{lemma}[theorem]{Lemma}
  \newtheorem{conjecture}[theorem]{Conjecture}
  \newtheorem*{conjecture*}{Conjecture}

  \newtheorem{question}[theorem]{Question}

  \theoremstyle{definition}
  \newtheorem{definition}[theorem]{Definition}

  \newtheorem*{claim*}{Claim}

  \newtheorem*{question*}{Question}
  \newtheorem*{answer*}{Answer}
  \newtheorem*{application*}{Application}
  \newtheorem*{ideas*}{ideas}
  \theoremstyle{remark}
  \newtheorem{remark}[theorem]{Remark}
  \newtheorem*{remark*}{Remark}

  \DeclareMathOperator{\Mod}{Mod}

  \DeclareMathOperator{\diam}{diam}

  \newcommand{\la}{\langle} 
  \newcommand{\ra}{\rangle}

  \newcommand{\param}{{\mathchoice{\mkern1mu\mbox{\raise2.2pt\hbox{$
  \centerdot$}}
  \mkern1mu}{\mkern1mu\mbox{\raise2.2pt\hbox{$\centerdot$}}\mkern1mu}{
  \mkern1.5mu\centerdot\mkern1.5mu}{\mkern1.5mu\centerdot\mkern1.5mu}}}
\renewcommand{\setminus}{{\smallsetminus}}

  \newcommand{\co}{\colon\thinspace}

   \newcommand{\F}{{\mathbb{F}}}

\renewcommand{\F}{\ensuremath{\mathbb{F} } }

\begin{document}


\title[Non-existence of Cannon--Thurston maps for HHGs]{non-existence  of Cannon--Thurston maps for hierarchically hyperbolic groups}

\author {Funda G\"ultepe}\thanks{First author was partially
supported by NSF grant DMS-2137611.}
\address{Department of Mathematics and Statistics\\
 University of Toledo\\
 Toledo, OHIO}
\email{\url{funda.gultepe@utoledo.edu}}
\urladdr{https://sites.google.com/view/fundagultepe}

\author{Ravi Tomar}
\address{Beijing International Center for Mathematical Research, Peking University, No. 5 Yiheyuan Road Haidian District, Beijing, P.R.China 100871}
\email{\url{ravitomar547@gmail.com}}

\begin{abstract}
We prove that Cannon--Thurston maps do not exist for hyperbolic normal subgroups of hierarchically hyperbolic groups, both for Morse and hierarchically hyperbolic boundaries. This result includes a large class of normal subgroups of the mapping class group of a surface of genus $g\geq 2$. As a corollary, we also prove a non-existence result for subgroups of most free by cyclic groups.  Moreover, we prove that the visual boundary of a CAT(0) group with isolated flats is homeomorphic to the relatively hierarchically hyperbolic boundary of the space it acts on, thus recovering a main result of \cite{BGGGS}.

\vspace{0.5cm}

\end{abstract}

\maketitle

\setcounter{tocdepth}{2}
\tableofcontents
\section{Introduction}

For hyperbolic subgroups $H$  of a hyperbolic group $G $,  let $\mathcal C (H)$ and $\mathcal C(G)$ be Cayley graphs for $H$ and $G$, respectively, given by choosing a finite generating set for $H$ and extending it to a finite generating set for $G$.  The \emph{Cannon–Thurston map} for the pair $(H,G)$ \emph{exists} if there is a continuous extension 
$\widehat \iota: \hat \calC(H)\rightarrow \hat \calC(G)$ between compactifications that extends the inclusion $\iota: \mathcal C(H) \hookrightarrow \mathcal C(G)$. In this setting, the restriction $\partial \iota\co \partial H\rightarrow \partial G$ of $\hat \iota$ between Gromov boundaries is called the \emph{Cannon--Thurston map}.

Existence of the Cannon--Thurston map was proved in various situations (see \cite{Minsky-rigidity,ADP-CT,McMullen,BowCT1, MitraCT, MitraTrees,Mj1}), with Mj (previously Mitra) \cite{Mj1} proving the existence in general.  Baker and Riley (\cite{BakerRiley1}) gave the first example of a hyperbolic subgroup of a hyperbolic group with no continuous Cannon--Thurston map (see also \cite{HMS-ct-maps}). As a contrast, Cannon--Thurston maps might exist even under \emph{arbitrarily heavy distortion} of a free subgroup of a hyperbolic group (\cite{BakerRiley2}).

In this paper, we will use the term \emph{Cannon--Thurston map} in the extended sense; in our work $G$ is not necessarily hyperbolic. In this extended setting, the existence of continuous maps between boundaries can still be useful to understand hyperbolic directions in a non-hyperbolic group. Results vary in the extended setting as well: for example, by the work of Bowditch (\cite{Bow-97}), a Cannon--Thurston map for suitable free subgroups of a \emph{relatively} hyperbolic group exists (see also \cite{pal-rel-hyp-ct-extension}) while there are examples of non-existence in other settings (for example \cite{BGGGS, Mousley, charney-sisto-ct-map-morse}). In this paper we will focus on investigating (non) existence for non-hyperbolic  free by cyclic and hierarchically hyperbolic groups.

Inspired by special cube complexes and mapping class groups,  \emph{hierarchical hyperbolicity} is a notion of \emph{coarse non-positive curvature} developed by Behrstock, Hagen and Sisto (\cite{BHS1,BHS2}) by axiomatizing some of their properties in a unified way. The spaces that satisfy the axioms are called  \emph{hierarchically hyperbolic spaces (HHS)} and a  \emph{hierarchically hyperbolic group (HHG)} is a finitely generated group whose Cayley graph admits an equivariant (with respect to group multiplication) HHS structure.  The class of HHS/Gs includes 
mapping class groups and many CAT(0) cube complexes (including all universal covers of compact special cube complexes), along with Gromov hyperbolic spaces,  Teichm\"{u}ller space with any of the usual metrics, many Artin groups \cite{BHS1} and more. 

A \emph{free by cyclic group} is a group $\Gamma$ given by
\[\Gamma =\F\rtimes_{\phi}\mathbb Z= \big{\la} \F, t\co tft^{-1}=\phi(f), \,\forall f\in \F  \big{\ra}    \]
where $\phi$ is an outer automorphism of the non-abelian free group $\F$. Free by cyclic groups have many properties analogous to those of (virtually) fibered $3$-- manifolds (\cite{Kielak1, DowKapLeinfbc1, DowKapLeinfbc2}) and appear in many areas of mathematics. The geometry of the free by cyclic groups is determined, to a large extent, by the dynamical properties of $\phi$ (\cite{Br-00, Gh-23, BGGH}). 
In \cite{BGGH}, the authors characterize the \emph{hierarchical hyperbolicity} of polynomially growing (hence all) free by cyclic groups by an algebraic condition on intersections of maximal virtually $\F \times \mathbb Z$ subgroups, called \emph{unbranched blocks}.

For hyperbolic free by cyclic groups, Cannon--Thurston maps exist by the work of Mitra (\cite{MitraCT}). However, there are examples of subgroups of relatively hyperbolic free by cyclic groups for which the Cannon--Thurston map, in the general sense, does not exist (\cite{BGGGS}.)  In \cite{BGGH}, the authors  asked:%

\begin{question}\label{ques:ct-map-free-by-cyclic} Let $\Gamma$ be an unbranched mapping torus (free by cyclic group) and let $H<\Gamma$ be a HHG subgroup. When does the inclusion map $H \hookrightarrow  \Gamma=\F\rtimes_\phi\mathbb Z$  extend to a continuous map $ \partial_{HH}H \rightarrow \partial_{HH}\Gamma$ between HHG boundaries? 
\end{question}
Hierarchically hyperbolic groups are compactified using the hierarchical boundary (HHS boundary) from \cite{DHSbound}. (For the details, see Section \ref{sec:bdryHHS}). By the same work, the answer to the question is positive when $H$ is \emph{hierarchically quasiconvex} ([Definition 5.1 \cite{BHS2}]), thus prompting the following conjecture made in \cite{BGGH}: 
\begin{conjecture}\label{conj:no-cannon-thurston}    
 Let $G$ be an HHG and $H <G$ an infinite hyperbolic
normal subgroup of infinite index. Then the Cannon--Thurston map does not exist for $(H, G)$ unless either $G$ is hyperbolic or $H$ is hierarchically quasiconvex in some HHG structure on $G$.
\end{conjecture}
In this paper, we prove a non-existence criterion (Lemma \ref{lemma-nonexistence-ct}) for Cannon--Thurston maps (in general sense) inspired by, and similar to that of \cite{BGGGS} and prove the following theorem, which confirms the Conjecture \ref{conj:no-cannon-thurston} under extra hypotheses that $G$ is not quasiisometric to a product of unbounded spaces and $G$ contains a subgroup isomorphic to $\mathbb Z^2$.
\begin{theorem}\label{main-thm-1}
	Let $G$ be a non-hyperbolic HHG that is not quasiisometric to a product of unbounded spaces.
	If $G$ contains a free abelian subgroup of rank at least $2$ then, for an infinite hyperbolic normal subgroup $N$ of $G$, we have the following:
	\begin{enumerate}
	\item There does not exist a Cannon--Thurston map from the Gromov boundary of $N$ to the HHG boundary of $G$.
	\item There does not exist a Cannon--Thurston map from the Gromov boundary of $N$ to the Morse boundary of $G$.
	\end{enumerate}
\end{theorem}

We note that by a recent theorem of Martelli (\cite{MartelliFlatClosing}), there are non-hyperbolic groups which do not contain $\mathbb Z^2$. As a result, we still do not know whether or not we can drop the hypothesis, as we do not know whether non-hyperbolic HHGs contain  $\mathbb Z^2$. However, for \emph{free by cyclic} HHGs, we only need the hypothesis that they are not quasiisometric to direct products. We thus answer Question \ref{ques:ct-map-free-by-cyclic} negatively for hyperbolic normal subgroups of free by cyclic groups: 

\begin{corollary}\label{cor:non-existence-ct-free-by-cyclic}
	Let $\Gamma=\F\rtimes_{\phi}\ZZ$ be an unbranched free by cyclic group such that $\Gamma$ is not quasiisometric to a product of unbounded spaces. Let $N$ be a hyperbolic normal subgroup of $\Gamma$. If $\phi$ is not atoroidal (equivalently, when $\Gamma$ is not hyperbolic), the inclusion map $N\hookrightarrow  \Gamma=\F\rtimes_\phi\mathbb Z$  does not extend to a continuous map from the Gromov boundary of $N$ to the HHG boundary of $\Gamma$. 
\end{corollary}

Another direct corollary of our Theorem \ref{main-thm-1} is for the mapping class groups, which is captured in the following corollary, which says that there is no Cannon--Thurston map for hyperbolic normal subgroups of the mapping class groups.
\begin{corollary}\label{cor: mcg}
	Let $S_{g,n}$ denote the connected, orientable surface of genus $g$ and $n\geq 0$ punctures such that $3g-3+n\geq 2$. For any hyperbolic normal subgroup $N$ of Mod$(S_g^n)$, the inclusion $\iota\co N \hookrightarrow \Mod(S_g^n) $ does not extend to a continuous map from the Gromov boundary of $N$ to the HHG boundary of  $\Mod(S_g^n) $.
\end{corollary}
Examples for subgroups of  the mapping class groups included in this result are subgroups generated by Dahmani--Guirardel--Osin (\cite{DGO}) and free normal subgroups generated in \cite{CMM}.

Lastly, we have the following theorem, and as a corollary we recover the main theorem of \cite{BGGGS}: 
\begin{theorem}\label{main-thm-2}
	Let $G$ be a group acting geometrically on a CAT(0) space $X$ with isolated flats. Then, the identity map of $X$ extends to a $G$-equivariant homeomorphism from the CAT(0) boundary of $X$ to the relatively hierarchically hyperbolic boundary of $X$.
\end{theorem}
A \emph{hierarchical space} (HS) is a space with hierarchical structure, given in Section \ref{sec:HS} and an HS space is \emph{relatively} hierarchically hyperbolic if each element $U$ of the index set $\gothic S $ is either  $\sqsubseteq$-minimal or the associated coordinate space $\calC U$ is hyperbolic. (\cite[Definition 6.8]{BHS2})

Currently, we do not know the full answer to Question \ref{ques:ct-map-free-by-cyclic}.  We remark that when $\Gamma=\F\rtimes_\phi\mathbb Z$  is a non-hyperbolic free by cyclic group, Cannon--Thurston map does not exist for the \emph{Morse} boundaries (see \cite{Cordes-MorseSurvey} for a survey on Morse boundaries); since all free by cyclic groups are Morse local to global (\cite{KudPetyt}) and, by \cite{CordesRSZ}, Cannon--Thurston map does not exist for normal hyperbolic subgroups of Morse local-to-global groups. 

The more general question of for which subgroups of HHGs  Cannon--Thurston maps exist is also widely open. (see also [Question 1.1 \cite{ABR-structure-invariant}] and the discussion thereafter). When $H$ is a free right angled Artin subgroup of the mapping class group of a connected, oriented surface, the Cannon--Thurston map exists under certain conditions, by a result of Mousley \cite{Mousley}. However, in the same work, Mousley also proved the non-existence of these maps for general right angled Artin subgroups embedded in mapping class groups, under different embeddings.

As for Morse boundaries for other groups, authors in \cite{charney-sisto-ct-map-morse}  found a criterion for non-existence of Cannon--Thurston maps for normal subgroups of semi-direct products of groups using certain right angled Artin groups, where they also provide an example for the existence of a Cannon--Thurston map using a mapping torus of a right angled Artin group. However, we do not know whether or not their extensions are hierarchically hyperbolic. 

\subsection{Acknowledgements} We thank Mark Hagen for many talks regarding hierarchically hyperbolic spaces and their boundaries; we appreciate his generosity with his time and attention. We also thank Jason Behrstock for his support and encouragement. R. Tomar gratefully acknowledges the support of the postdoctoral fellowship from BICMR, Peking University.

\section{Background}
\subsection{Coarse geometry}
	In this subsection, we collect some basic notions that we shall use throughout the paper. Let $(X,d)$ be a metric space. For $K\geq 0$ and $A\subseteq X$, we denote by $N_K(A)$ the set $\{x\in X: d(x,A)\leq K\}$. For $x\in X$ and $R\geq 0$, $B_R(x)$ denotes the closed ball with center $x$ and radius $R\geq 0$. The Hausdorff distance between two subsets $A$ and $B$ of $X$ is denoted by $Hd_X(A,B)$. When $X$ is understood from the context, we will omit $X$ and simply write $Hd(A, B)$. For $E\geq 0$, the metric space $X$ is said to be $E$-quaisgeodesic if each pair of points of $X$ is joined by an $E$-quasigeodesic. 
    
    The following lemma is a basic exercise in point-set topology, so we skip its proof.
		\begin{lemma}\label{lemma-subseq-convergence}
			Suppose $Z$ is a Hausdorff topological space, $z\in Z$ and $\{z_n\}$ is a sequence in $Z$. 
			If for any subsequence $\{z_{n_k}\}_k$ of $\{z_n\}$, there exists a further subsequence $\{z_{n_{k_{l}}}\}_l$ 
			of $\{z_{n_k}\}_k$ such that $\{z_{n_{k_{l}}}\}$ converges to $z$ then $\lim z_n= z$. \qed
		\end{lemma}
\subsection{Hierarchically hyperbolic  spaces}\label{sec:HS}
In the following definition, we use the notation $d_W(-,-)$ to denote distance
in a space $\calC W$, where $W$ is in an index set $\gS$. We will follow this convention where it will not cause confusion. Similarly, where it will not cause confusion, we
write, for example, $d_W(x,y)$ to mean $d_W(\pi_W(x), \pi_W(y))$, where $x,y\in \calX$, and $W\in\gS$, and
$\pi_W:\calX\to \calC W$ is a projection.

\begin{definition}[Hierarchical space, (relatively) hierarchically hyperbolic space]
    Let $E>0$ be a constant and $\calX$ be a $(E,E)$-quasigeodesic space. The space $\calX$ is said to be {\em hierarchical space (HS)} if there exists an indexing set $\gS$ and a family of geodesic spaces $\{\calC U:U\in\gS\}$ such that the following conditions are satisfied:
		\begin{enumerate}
			\item ({\bf Projection}) For each $U\in\gS$ there exists a {\em projection} $\pi_U:\calX\to 2^{\calC U}$ such that, for each $x\in \calX$, $\pi_U(x)\neq \phi$ and $\diam (\pi_U(x))\leq E$. Moreover,  each $\pi_U$ is $(E,E)$-coarsely Lipschitz map $\pi_U:\calX\to \calC U$ and $E$-coarsely surjective.
			
			\item ({\bf Nesting}) if $\gS\neq \phi$, then there is a partial order $\sqsubseteq$ on $\gS$ such that $\gS$ has a unique $\sqsubseteq$-maximal element. When $V\sqsubseteq W$, we say that $V$ is {\em nested} in $W$. For $W\in \gS$, $\gS_W$ denote the set of all element of $\gS$ which are nested in $W$. For $V\sqsubsetneq W$, there is a subset $\rho_W^V\subseteq \calC W$ such that $\diam(\rho_W^V)\leq E$. There is also a projection $\rho_V^W:\calC W\to 2^{\calC V}$.
			
			\item ({\bf Orthogonality}) There exists a symmetric and antireflexive relation on $\gS$ called orthogonality; we write $V\perp W$ if they are orthogonal. Also, whenever $U\sqsubseteq V$ and $V\perp W$, we require that $U\perp W$. Finally, if $V\perp W$ then $V$ and $W$ are not $\sqsubseteq$-comparable. We denote $\gS_W^{\perp}$ the set of all $V\in\gS$ such that $V\perp W$. 
			
			\item ({\bf Containers}) We require that, for each $T\in \gS$ and $U\in \gS_T$, if $\{V\in\gS_T:V\perp U\}\neq \phi$ then there exists $W\in\gS_T\setminus\{T\}$ such that $V\sqsubseteq W$ whenever $V\in\gS_T$ and $V\perp U$. We call $W$ the {\em container} of $U$ in $T$.
			
			\item ({\bf Transversality}) If $V,W\in \gS$ are not orthogonal and neither is nested in the other, then we say that $V$ and $W$ are {\em transversal}, denoted by $V\pitchfork W$. Moreover, if $V\pitchfork W$ then there exist nonempty sets $\rho_V^W\subseteq \calC V$ and $\rho_W^V\subseteq \calC W$, each of which is of diameter at most $E$.
			
			\item ({\bf Consistency})  For all $x\in\calX$ and $U,V,W\in \gS$:
			\begin{itemize}
				\item If $V\pitchfork W$, then $\min\{d_V(x,\rho_V^W),d_W(x,\rho_W^V)\}\leq E$.
				\item If $V\sqsubseteq W$, then $\min\{d_W(x,\rho_W^V),\diam_{\calC V}(\pi_V(x)\cup\rho_V^W(\pi_W(x)))\}\leq E.$
				\item If $U\sqsubseteq V$ and either $V\pitchfork W$ or $V\sqsubseteq W$ and $U\not\perp W$, then $d_W\{\rho_W^U,\rho_W^V\}\leq E$.
			\end{itemize}
			
			\item ({\bf Finite complexity}) Any set of pairwise $\sqsubseteq$-comparable elements of $\gS$ has cardinality at most $E$. 
			
			\item ({\bf Bounded geodesic image}) For all $V,W\in\gS$ with $V\sqsubsetneq W$, and for all geodesics $\gamma$ of $ \calC W$, we require that either $\gamma\cap \mathcal N_E(\rho_W^V)\neq \phi$ or $\diam _{\calC V}(\rho_V^W(\gamma))\leq E$.
			
			\item ({\bf Partial realization}) If $\{V_j\}$ is a finite collection of pairwise orthogonal elements of $\gS$ and $p_j\in \pi_{V_j}(\calX)$, then there exists an $x\in \calX$ such that:
			\begin{itemize}
				\item $d_{V_j}(x,p_j)\leq E$.
				\item For each $j$ and each $W$, if $V_j\sqsubsetneq W$ or $V_j\pitchfork W$, $d_W(x,\rho_W^{V_j})\leq E$.
			\end{itemize}
			
			\item ({\bf Large link}) For all $W$ and $x,y\in \calX$, there exists $\{T_i\}_{i=1}^m\subset\gS_W\setminus\{W\}$, where $m$ is at most $Ed(x,y)+E$, such that for all $T\in\gS_W\setminus\{W\}$ either $T\sqsubseteq T_i$ for some $i$ or $d_T(x,y)\leq E$. Moreover, $d_W(x,\rho_W^{T_i})\leq m$ for all $i$.
			
			\item ({\bf Uniqueness}) There exists a function $\theta:[0,\infty)\to[0,\infty)$ such that, for all $r\geq 0$, if $x,y\in\calX$ and $d(x,y)\geq \theta(r)$, then there exists $V\in\gS$ such that $d_V(x,y)\geq r$.
		\end{enumerate}
\end{definition}
	
We often refer to $\gS$, together with the nesting and orthogonality relations, and the projections as a {\em hierarchical structure} for the space $\calX$. We call the elements of $\gS$ the \emph{domains} of $\gS$ and call $\rho_W^V$ the {\em relative projection} from $V$ to $W$. The number $E$ is called the {\em hierarchy constant} for $\gS$. We write $(\calX,\gS)$ to denote the hierarchical space with the hierarchical structure $\gS$. If $\calC U$ is $E$-hyperbolic for all $U\in\gS$, then $(\calX,\gS)$ is {\em hierarchically hyperbolic}. If $\calC U$ is $E$-hyperbolic for all non-$\sqsubseteq$-minimal $U\in\gS$, then $(\calX,\gS)$ is {\em relatively hierarchically hyperbolic}.

A {\em hierarchical group (HG)} is a finitely generated group $G$ that acts on a hierarchical space $(\calX,\gS)$ such that
\begin{enumerate}
	\item the action of $G$ on $\calX$ is geometric,
	\item  $G$ acts on $\gS$ by a $\sqsubseteq$-,$\perp$-, and $\pitchfork$-preserving bijection, and $\gS$ has finitely many $G$-orbits,
	\item the action is compatible with the hierarchical structure on $\calX$, see \cite[p. 5]{petyt-spriano}, \cite[Subsection 1G]{BHS2}.
    
\end{enumerate} 

In this case, we say that $\gS$ is a hierarchical structure for the group $G$ and use the pair $(G,\gS)$ to denote the hierarchical group $G$ equipped with the specific hierarchically hyperbolic group structure $\gS$ (For a precise definition, see \cite[Subsection 1.2.3]{BHS-asymptotic-dim}). From condition (1) in the definition of an HHG, it follows that $G$ is quasiisometric to $\calX$. From here, one can give a hierarchical structure on the Cayley graph of $G$, and the left action of $G$ on the Cayley graph also satisfies all the conditions for being an HHG. Thus, to define an HHG, one can use a Cayley graph of $G$ itself. Also, it is easy to show that the definition of HHG does not depend on the choice of the Cayley graph. Suppose $(\calX,\gS)$ is a hierarchically hyperbolic space (HHS) or a relative hierarchically hyperbolic space (RHHS). In that case, we say that $(G,\gS)$ is a {\em hierarchically hyperbolic group} (HHG) or a {\em relative hierarchically hyperbolic group} (RHHG).

\begin{remark}
    \label{example-relhhg}
    Let $X$ be a CAT(0) space with isolated flat property. Then, by Theorem \ref{theorem-cat(0)-isolated-flat}, $X$ is hyperbolic relative to a collection of flats $\calF$. We now describe a relative HHS structure on $X$.

    \begin{itemize}
        \item {\bf Index set.} Let $\hat{X}$ denote the coned-off space with respect to $\calF$. Define $$\gS:=\calF \cup\hat{X}.$$
        \item {\bf Relations.} For all $F\neq F'\in\calF$, we declare $F\pitchfork F'$, and $F\sqsubsetneq \hat{X}$ for all $F\in \calF$ (the orthogonality relation is empty and there is no other nesting).
        \item For $F\in\calF$, we define $\calC F=F$, and for $\hat{X}$, we declare $\calC \hat{X}=\hat{X}$.
        \item {\bf Projections.} The projection $\pi_{\hat{X}}:X\to\hat{X}$ is the inclusion. For $F\in \calF$, we define $\pi_F:X\to F$ to be the closest point projection. Note that the projection maps are coarsely surjective.
        \item For $F\neq F'$, we define $\rho_F^{F'}=\pi_F(F')$ (since $X$ is hyperbolic relative to $\calF$, this is uniformly bounded).
        \item For $F\sqsubsetneq \hat{X}$, define $\rho_{\hat{X}}^F$ to be the cone point corresponding to $F$.
        \item For each $F\in\calF$, let $\rho_F^{\hat{X}}:\hat{X}\to F$ be defined by $\rho_F^{\hat{X}}(x)=\pi_F(x)$ for $x\in X$, while $\rho_F^{\hat{X}}(x)=\pi_F(F')$ whenever $x$ lies in the cone on $F'\in\calF.$
    \end{itemize}
    By \cite[Theorem 9.3]{BHS2}, we see that $(X,\gS)$ is a relatively hierarchically hyperbolic space.
\end{remark}

\subsection{Boundaries of spaces and groups}

We do not recall the definitions of Gromov and Morse boundaries here; instead, we will list properties we use inside proofs. For the details, we refer the reader to the survey \cite{Kapovich-bdry} for the Gromov boundary and \cite{Cordes-MorseSurvey} for the Morse boundary. 

\subsubsection{Boundaries of CAT(0) spaces and groups}
For the definition of CAT(0) spaces and their properties, one is referred to \cite{BH}. A {\em flat} in a CAT(0) space is an isometrically embedded copy of $\EE^n$ for $n\geq 2$.
\begin{definition}\label{defn;cat(0)-groups-with-isolated-flats}
    Let $G$ be a group that acts geometrically on a CAT(0) space $X$. The space $X$ has {\em isolated flats property} if $X$ is not quasiisometric to $\EE^n$, and there is a $G$-invariant collection of flats $\calF$ satisfying the following:
    \begin{enumerate}
        \item There exists a constant $D\geq 0$ such that each flat in $X$ lies in a $D$-neighborhood of some $F\in \calF$.
        \item For every $r\geq 0$ there exists $k(r)$ such that, for any two distinct flats $F,F'\in\calF$, $\diam_{X}(N_r(F)\cap N_r(F'))\leq k(r)$.
    \end{enumerate}
    
    A group $G$ is said to be CAT(0) {\em with isolated flat property} if $G$ acts geometrically on a CAT(0) space with isolated flat property. 
\end{definition}
\begin{theorem}\cite[Theorem 1.2.1, 1.2.2]{hruska-kleiner-isolated-flat}\label{theorem-cat(0)-isolated-flat}
    Let $G$ be a group that acts geometrically on a CAT(0) space $X$. 
    
    (1) The following are equivalent:
    \begin{enumerate}[label=(\alph*)]
        \item The space $X$ has the isolated flat property.
        \item The space $X$ is hyperbolic relative to a collection of flats.
        \item The groups $G$ are hyperbolic relative to a collection of virtually free abelian groups of rank at least $2$.
    \end{enumerate}

    (2) If $X$ has the isolated flat property, then the visual boundary of $X$ is a group invariant of $G$.
\end{theorem}
For a proper CAT(0) space $X$, we denote the visual boundary of $X$ by $\partial X$. Define $\bar{X}=X\cup\partial X$ and fix a base point $x_0\in X$. For $x\in \bar{X}$, we denote by $\gamma_x$ the geodesic joining $x_0$ and $x$. Then, the topology on $\bar{X}$ is defined as follows. Given $r>0$ and $\epsilon>0$, a neighborhood at a point of $\partial X$ represented by a geodesic ray $c$ is given by 
$$U(c,r,\epsilon)=\{x\in\bar{X}: d(x_0,x)\geq r \text{ and } d(\gamma_x(r),c(r))\leq \epsilon\}.$$ Then, the set of open ball in $X$ together with the collection $U(c,r,\epsilon)$, where $c$ is a geodesic ray with $c(0)=x_0$, is a basis for the topology on $\bar{X}$. One is referred to \cite[Chapter II]{BH} for the details.

\begin{lemma}\label{lemma-cat(0)-nbhd}
    Given $p\in\partial X$, there exists a sequence of neighborhoods $\{K_n\}$ of $p$ such that, for any other neighborhood $K$ of $p$, $K_n\subseteq K$ for all sufficiently large $n$.
\end{lemma}
\begin{proof}
    It is sufficient to prove the lemma for the neighborhood basis defined above around $p$. Let $c$ be the geodesic ray that starts at $x_0$. Define $K_n:=U(c,n,1/n).$ Let $U(c,r,\epsilon)$ be any other neighborhood basis around $p$. Choose $n$ such that $1/n<\epsilon$ and $n>r$. Then, it is easy to check that $K_m\subseteq U(c,r,\epsilon)$ for all $m\geq n$.
\end{proof}
The following lemma easily follows from the definition of the convergence of a sequence in a CAT(0) space. 
\begin{lemma}\label{lemma:weakly-visible}
    Let $X$ be a CAT(0) space. Let $\{x_n\}$ be a sequence of points in $X$ that converges to a point $\xi\in\partial X$. If $\{y_n\}$ is a sequence in $X$ such that $d(x_n,y_n)$ is uniformly bounded for all $n$, then $\{y_n\}$ also converges to $\xi$.   \qed
\end{lemma}
    
\subsubsection{Boundary of hierarchically hyperbolic spaces }\label{sec:bdryHHS}
In \cite{DHSbound}, the authors introduced the notion of a boundary of an HHS. From \cite[Section 2]{DHSbound}, we recall the definition of the hierarchical boundary and its topology. For $U\in \gS$, $\partial\calC U$ denotes the Gromov boundary \cite{gromov-essay} of $\calC U$.
\begin{definition}
	Let $(\calX,\gS)$ be an HHS. A subset $\overline{S}\subset\gS$ is said to be a {\em support set} if $S_i\perp S_j$ for all $S_i,S_j\in \overline{S}$. Given a support set $\overline{S}$, the {\em boundary point with support $\overline{S}$} is a formal sum $p=\sum_{S\in\overline{S}} a_S^pp_S$, where $p_S\in\partial\calC S$, $a_S>0$, and $\sum_{S\in\overline{S}}a_S^p=1$. By \cite[Lemma 2.1]{BHS2}, such sums are necessarily finite. We denote the support of the boundary point $p$ by Supp($p$). The {\em hierarchical boundary} $\partial(\calX,\gS)$ of $(\calX,\gS)$ is the set of all boundary points.
\end{definition}
When the specific HHS structure is clear, we write $\partial_{HH}\calX$ instead of $\partial(\calX,\gS)$. For an HHG $(G,\gS)$, let $(\calX,\gS)$ be a corresponding HHS. Then, the hierarchical boundary $\partial(G,\gS)$ of $(G,\gS)$ is defined to the hierarchical boundary of $(\calX,\gS)$. Similarly, when the HHG structure on $G$ is clear from the context, we denote the HHG boundary of $G$ by $\partial_{HH}G.$

{\bf Topology on $\partial_{HH}\calX$.} Before defining the topology, we need the notion of a remote point and boundary projection.

\begin{definition}
	[Remote point] A point $q\in\partial_{HH}\calX$ is called a {\em remote point} with respect to a support set $\overline{S}$ if Supp($q$)$\cap\overline{S}=\emptyset$ and, for all $S\in\overline{S}$, there exists $T_S\in$ Supp($q$) such that $S\notperp T_S$. The set of all remote points of $\partial_{HH} \calX$ with respect to $\overline{S}$ is denoted by $\partial^{rem}_{\overline{S}}(\calX)$.
\end{definition}

For a support set $\overline{S}$, we denote $\overline{S}$$^{\perp}$ the set of all $U\in\gS$ such that $U\perp V$ for all $V\in\overline{S}$. Given a support set $\overline{S}$ and $q\in\partial^{rem}_{\overline{S}}\calX$, let $\overline{S}_q$ denote the union of $\overline{S}$ and the set of all $U\in\overline{S}$$^\perp$ such that $U$ is not orthogonal to some $T_U\in$ Supp($q$).
\begin{definition}
	[Boundary projection]\label{defn:boudnary-projection}  Define a {\em boundary projection} $\partial\pi_{\overline{S}}(q)\in \prod_{S\in \overline{S}_q}\partial\calC S$
	as follows. Let $q=\sum_{T\in\text{Supp}(q)}a_T^q q_T$. For each $S\in \overline{S}_q$ , let $T_S\in\text{Supp}(q)$ be chosen so that $S$ and $T_S$ are not orthogonal. Define the
	$S$-coordinate $(\partial\pi_{\overline{S}}(q))_S$ of $\partial\pi_{\overline{S}}(q)$ as follows:
	\begin{enumerate}
		\item  If $T_S\sqsubseteq S$ or $T_S\pitchfork S$, then $(\partial\pi_{\overline{S}}(q))_S=\rho_{S}^{T_S}$.
		\item  Otherwise, $S\sqsubseteq T_S$. Choose a $(1,20E)$–quasigeodesic ray $\gamma$ in $\calC T_S$ joining $\rho_{T_S}^S$ to $q_{T_S}$. By the bounded geodesic image axiom, there exists $x\in\gamma$
		such that $\rho_{S}^{T_S}$
		is coarsely constant on the subray of $\gamma$ beginning at $x$. Let
		$(\partial\pi_{\overline{S}}(q))_S=\rho_{S}^{T_S}(x)$.
	\end{enumerate}
\end{definition}
The map $\partial\pi_{\overline{S}}$ is coarsely independent of the choice of $\{T_S\}_{S\in\overline{S}}$ (see \cite[Lemma 2.1]{DHSbound}). Now, we are ready to define the topology on $\partial_{HH}\calX$.

Fix a base point $x_0\in\calX$. We define a neighborhood basis for each point $p=\sum_{S\in\overline{S}}a_S^pp_S$, where $p_S\in\partial\calC S$ for each $S\in\text{Supp}(p)=\overline{S}$. For each $S\in\gS$, choose a neighborhood $U_S$ of $p_S$ in $\calC S\cup\partial\calC S$, and choose $\epsilon>0$. Now, we define the following three subsets of $\partial_{HH}\calX$ which contribute in the definition of a neighborhood around $p$.
\begin{definition}
	[Remote part] The {\em remote part} $\mathcal{B}^{rem}_{\{U_S\},\epsilon}(p)$ is the set of all $q\in\partial_{HH}\calX$ such that the following hold:
	\begin{enumerate}
		\item For all $S\in\overline{S}$, $(\partial\pi_{\overline{S}}(q))_S\in U_S$.
		\item $\sum_{T\in\overline{S}^{\perp}}a_T^q<\epsilon$.
		\item For all $S\in \overline{S}_q$ and $S'\in\overline{S}$, $\bigg|\dfrac{d_S(x_0,(\partial\pi_{\overline{S}}(q))_S)}{d_{S'}(x_0,(\partial\pi_{\overline{S}}(q))_{S'})}-\dfrac{a_S^p}{a_{S'}^p}\bigg|<\epsilon$.
	\end{enumerate} 
\end{definition}
\begin{definition}
	[Non-remote part] The {\em non-remote part} $\mathcal{B}^{non}_{\{U_S\},\epsilon}(p)$ is the set of points $q=\sum_{T\in\text{Supp}(q)}a_T^q q_T\in \partial_{HH}\calX-\partial^{rem}_{\overline{S}}\calX$ such that the following hold, where $A=\overline{S}\cap\text{Supp}(q)$:
	\begin{enumerate}
		\item For all $T\in A$, $q_T\in U_T$.
		\item For all $T\in A$, $|a_T^p-a_T^q|<\epsilon.$
		\item $\sum_{V\in\text{Supp}(q)- A}a_V^q<\epsilon$.
	\end{enumerate}
\end{definition}
\begin{definition}
	[Interior part] The {\em interior part} $\mathcal{B}^{int}_{\{U_S\},\epsilon}(p)$ is the points $x\in\calX$ such that the following conditions are satisfied:
	\begin{enumerate}
		\item For all $S\in\overline{S}$, $\pi_S(x)\in U_S$.
		\item For all $S,S'\in\overline{S}$, $\bigg|\dfrac{d_S(x_0,\pi_S(x))}{d_{S'}(x_0,\pi_{S'}(x)}-\dfrac{a_S^p}{a_{S'}^p}\bigg|<\epsilon.$
		\item For all $S\in\overline{S}$ and $T\in \overline{S}^{\perp}$, $\dfrac{d_T(x_0,x)}{d_S(x_0,x)}<\epsilon$.
	\end{enumerate}
\end{definition}
\begin{definition}[Topology on $\calX\cup\partial_{HH}\calX$]\label{def:toporhh}
	 For each $p\in\partial_{HH}\calX$ with support $\overline{S}$, and $\{U_S\}_{S\in\overline{S}}$, $\epsilon>0$ as above, let $$B_{\{U_S\},\epsilon}(p):=\mathcal{B}^{rem}_{\{U_S\},\epsilon}(p)\cup \mathcal{B}^{non}_{\{U_S\},\epsilon}(p)\cup \mathcal{B}^{int}_{\{U_S\},\epsilon}(p).$$ We declare the set of all such $B_{\{U_S\},\epsilon}(p)$ to form a neighborhood basis around $p$. Also, we include the open subsets of $\calX$ in the topology of $\calX\cup\partial_{HH}\calX$.
\end{definition}

\begin{theorem}\label{theorem:property-hhs-bdry}
	\textup{(\cite[Theorem 3.4]{DHSbound})} If $\calX$ is proper, then $\overline{\calX}:=\calX\cup\partial_{HH}\calX$ is a compact metrizable space. Moreover, $\calX$ is dense in $\overline{\calX}$.
\end{theorem}

\begin{definition}[Limit set]\label{definition:limit-set}
	Let $\calY$ be a subset of an HHS $(\calX,\gS)$. Then, the {\em limit set $\Lambda_{\calX}(\calY)$} of $\calY$ is defined as the set $\{p\in\partial_{HH}\calX:\text{there exists a sequence $\{y_n\}$ in $\calY$ such that $\{y_n\}$ converges to $p$}\}.$
\end{definition}

Throughout the paper, whenever a group $G$ has a boundary, for a subgroup $H$ of a group $G$, we denote by $\Lambda_G(H)$ the limit set of $H$ in the boundary of $G$; which we will define same way as in Definition \ref{definition:limit-set}. When $G$ is clear from the context, we  will write $\Lambda(H)$.

\begin{remark}
    Let $(G,\gS)$ be an HHG such that $G$ is virtually cyclic. Then, $G$ is a hyperbolic group, and hence $\partial_{HH}(G)$ has two points (\cite[Theorem 4.3]{DHSbound}). In fact, $G$ is a locally quasiconvex group, i.e. every finitely generated subgroup of $G$ is quasiconvex in $G$. Hence, every finitely generated subgroup of $G$ admits a Cannon-Thurston map. So, throughout the paper, we assume that HHGs are not virtually cyclic.
\end{remark}
For an alternate proof of the following lemma, see \cite[Corollary 5.12]{ABR-structure-invariant})
\begin{lemma}\label{lemma:limit-set-of-normal-subgroup}
    Let $(G,\gS)$ be an infinite HHG such that the hyperbolic space corresponding to the maximal domain of $\gS$ is unbounded. Let $N$ be an infinite subgroup of $G$. Then, we have the following:
    \begin{enumerate}
        \item $\Lambda(N)$ is non-empty.
        \item If $N$ is normal in $G$, then $\Lambda(N)=\partial_{HH}G$
    \end{enumerate}
\end{lemma}
\begin{proof}
    (1) Since $N$ is infinite, there exists an unbounded sequence $\{h_n\}\subset N$. By Theorem \ref{theorem:property-hhs-bdry}, $\overline{G}$ is compact. Hence, $\{h_n\}$ converges to a point in $\partial_{HH}(G)$. Hence, $\Lambda(N)$ is non-empty. 

    (2) Let $S$ be the unique $\sqsubseteq$-maximal element of $\gS$. By our assumption, the hyperbolic space $\calC S$ is unbounded. 
     Since $G$ acts on itself coboundedly, $G$ acts coboundedly on $\calC S$. Since the orbit of $G$ for the action on $\calC S$ is cobounded in $\calC S$ the limit set $\Lambda_{\calC S}(G)$ is equal to $\partial \calC S$; the Gromov boundary of $\calC S$. Since $N$ is a normal subgroup of $G$, the limit set $\Lambda_{\calC S}(N)$ of $N$ in $\calC S$ is $G$-invariant. This implies that $\Lambda_{\calC S} (N)=\partial \calC S$. Note that $\Lambda_{\calC S}(N)\subset \Lambda (N)$. Since, by \cite[Corollary 6.30]{DHSbound}, $\partial \calC S$ is dense in $\partial_{HH}G$, the limit set of $N$ in $\partial_{HH}G$ is equal to $\partial_{HH}G$. 
\end{proof}

\subsection{Cannon--Thurston maps} We define Cannon--Thurston maps in a more general setting. Let $G$ be a finitely generated group and let $\Gamma(G,A)$ be the Cayley graph of $G$ with respect to a finite generating set $A$. Suppose $\overline{G}:=\Gamma(G,A)\cup \partial G$ is a compactification of $\Gamma(G,A)$, i.e. the inclusion $\Gamma(G,A)\to \overline{G}$ is a topological embedding whose image is a dense open subset of $\overline{G}$. We define $\overline{G}$ as the {\em compactification} of $G$. We also assume the following:
\begin{enumerate}
    \item $\overline{G}$ is metrizable, and it does not depend on the choice of the Cayley graph of $G$. 

    \item Each element $g\in G$ gives rise to a homeomorphism from $\partial G\to\partial G$. 
\end{enumerate}

When $G$ is hyperbolic (resp. HHG), such a compactification of $G$ is given by the Gromov boundary of $G$ (resp. HHG boundary or Morse boundary) of $G$.

\begin{definition}\label{defn:CT-maps}
    Let $H$ be a finitely generated subgroup of a finitely generated group $G$. Suppose $H$ and $G$ admit compactifications as defined above. 
    Let $h_0\in H$ and $g_0\in G$. 
    A {\em Cannon--Thurston (CT)} map for the pair $(H,G)$ is the continuous map $ f:\partial H\to \partial G$ defined by the following:
    $$f(\lim_{i\to\infty}h_ih_0)=\lim_{i\to\infty}h_ig_0$$ provided that $f(\xi)$ does not depend on the sequence $\{h_i\}$ such that $h_ih_0\to \xi$ as $i\to\infty.$
\end{definition}

\begin{remark}
    (1) By definition, a Cannon--Thurston map, if it exists, is unique and $H$-equivariant.

    (2) When $H$ and $G$ are hierarchically hyperbolic, the Cannon--Thurston map for the pair $(H,G)$ is defined similarly by replacing $\partial H$ and $\partial G$ by $\partial_{HH}H$ and $\partial_{HH}G$, respectively.

    (3) The Cannon--Thurston map for the pair $(H,G)$ with respect to Morse boundaries is defined similarly by replacing $\partial H$ and $\partial G$ with the Morse boundaries of $H$ and $G$, respectively. Although the Morse boundary of a group need not give a compactification of the group, the definition of CT maps for Morse boundaries still makes sense.

\end{remark}
We now record the following simple lemma that is relevant to us.
\begin{lemma}\label{lemma:image of CT is the limit set}
    Let $H<G$ be infinite finitely generated groups admitting compactifications as above. If the pair $(H,G)$ admits a Cannon--Thurston map $f:\partial H\to\partial G$, then $f(\partial H)=\Lambda_G(H).$
\end{lemma}
\begin{proof}
    Let $\xi\in\partial H$. Since $H$ is dense in $\overline{H}=H\cup\partial H$, there exists a sequence $\{h_n\}$ in $H$ such that $h_n\to\xi$ as $n\to\infty.$ Then, by the definition of the Cannon--Thurston map, $h_n\to f(\xi)$. This implies that $f(\xi)\in\Lambda_G(H)$, and thus $f(\partial H)\subseteq \Lambda_G(H)$. Conversely, let $\eta\in\Lambda_G(H).$ Then, there exists a sequence $\{h_n\}$ in $H$ such that $h_n\to\eta$ as $n\to\infty.$ Since $\overline{H}$ is compact, there exists subsequence $\{h_{n_k}\}$ of $\{h_n\}$ such that $h_{n_k}\to\xi$ as $k\to\infty$ for some $\xi\in\partial H$. Now, it follows from the definition of Cannon--Thurston map that $f(\xi)=\eta.$ Hence, $\eta\in f(\partial H)$. This completes the proof of the lemma.
\end{proof}
We end this subsection with the following non-existential criterion for CT maps. Before that, we include the following definition for the sake of completeness.

\begin{definition}
     Let $Z$ be a topological space. An element $g\in Homeo(Z)$ acts with North-South dynamics if $g$ fixes two points $z^{+}\neq z^{-}\in Z$ and for any open subsets $U,V$ of $Z$ with $z^{-}\in U$ and $z^{+}\in V$ there exists an integer $N$ so that $g^n(Z\setminus U)\subset V$ and $g^{-n}(Z\setminus V)\subset U$ for all $n\geq N$.
\end{definition}

\begin{lemma}[Criterion for non-existence of CT maps]\label{lemma-nonexistence-ct}
    Let $H<G$ be infinite finitely generated groups that admit compactifications such that the following hold: There exists $h\in H$ such that, 
    \begin{enumerate}
        \item $h$ acts with North-South dynamics on $\partial H$.
        \item $h$ fixes pointwise a subset $L\subset\partial G$ such that $|L|\geq 3$ and $L\subseteq\Lambda_G(H)$, where $\Lambda_G(H)$ denotes the limit set of $H$ in $\partial G.$
    \end{enumerate}
    Then, the inclusion $i:H\hookrightarrow G$ does not extend to an $H$-equivariant continuous map $\partial i:\partial H\to\partial G.$
\end{lemma}
\begin{proof}
   The idea of the proof is from Lemma 5.1 of \cite{BGGGS}. Suppose such a continuous map $\partial i$ exists. Then, by Lemma \ref{lemma:image of CT is the limit set}, $\partial i(\partial H)=\Lambda_G(H).$ Let $h^{\pm}$ be the fixed points of $h$ in $\partial H.$ Let $x=\partial i(y)\in\partial i(\partial H)$ be distinct from $\partial i(h^{\pm})$. Now, $\lim_{n\to\infty}h^nx=\lim_{n\to\infty}h^n\partial i(y)=\partial i(\lim_{n\to\infty}h^ny)=\partial i(h^+).$ Hence, $h$ fixes at most two points of $\Lambda_G(H)$. This is a contradiction as $h$ fixes pointwise $L$.
\end{proof}
\section{Non-Existence of Cannon--Thurston maps for hierarchically hyperbolic groups}
This section is devoted to proving the main results of this paper. We start here with the following, which is proved using the criterion for non-existence of CT maps (Lemma \ref{lemma-nonexistence-ct}).

\begin{theorem}\label{theorem-main}
    Let $(G,\gS)$ be a non-hyperbolic HHG such that the hyperbolic space corresponding to the maximal domain of $\gS$ is unbounded (equivalently, $G$ contains a Morse element, equivalently $G$ is not quasiisometric to a product of unbounded spaces). Let $N$ be an infinite hyperbolic normal subgroup of $G$. Suppose $G$ contains a free abelian subgroup $K$ of rank at least $2$ such that $K\cap N\neq\{1\}.$ Then, the inclusion $N\hookrightarrow G$ does not extend to  a Cannon--Thurston map from the Gromov boundary of $N$ to the HHG boundary of $G$.
\end{theorem}
\begin{proof}
To prove the theorem, we check the conditions of Lemma \ref{lemma-nonexistence-ct}. Let $H_K$ denote the hierarchically quasiconvex hull of $K$ in $G$ (\cite[Definition 6.1]{BHS2}). Then, by \cite[Proposition 2.17]{HHSequi}, 
up to taking powers, the elements of $K$ fix pointwise the limit set $\Lambda H_K$ of $H_K$ in $\partial_{HH}G$. Let $h\in N\cap K$. Then, $h$ is an infinite order element of $N$, and hence it has North-South dynamics on the Gromov boundary $\partial N$  of $N$. Note that, by Lemma \ref{lemma:limit-set-of-normal-subgroup} (2), $\Lambda (N)=\partial_{HH}G$. 
Assume that there exists a Cannon--Thurston map $f:\partial N \to\partial_{HH}G.$ Then, by Lemma \ref{lemma:image of CT is the limit set}, $f(\partial N)=\Lambda (N)$, and hence $f$ is surjective. Thus, $\Lambda H_K\subset f(\partial N)$ and hence the element $h$ satisfies both the conditions of Lemma \ref{lemma-nonexistence-ct}. Thus, we get a contradiction to our assumption of the existence of $f$ since unbounded domains of $H_K$ are quasilines, and as a result $\Lambda(H_K)$ will be homeomorphic to $S^l$. 
\end{proof}

The proof of the following lemma is included in the proof of  Proposition 5.2 of \cite{BGGGS}.

\begin{lemma}\label{lemma-normal-intersect-freeabelian}
    Let $G$ be a group that contains $\ZZ^2$. Then, for any hyperbolic normal subgroup $N$ of $G$ there exists a group $K$ isomorphic to $\ZZ^2$ such that $N\cap K\neq \{1\}$. 
\end{lemma}

We are ready to prove our main Theorem \ref{main-thm-1} which proves Conjecture \ref{conj:no-cannon-thurston} under extra hypotheses.

\begin{mainthm1}
	Let $G$ be a non-hyperbolic HHG that is not quasiisometric to a product of unbounded spaces.
	If $G$ contains a free abelian subgroup of rank at least $2$ then, for an infinite hyperbolic normal subgroup $N$ of $G$, we have the following:
	\begin{enumerate}
	\item There does not exist a Cannon--Thurston map from the Gromov boundary of $N$ to the HHG boundary of $G$.
	\item There does not exist a Cannon--Thurston map from the Gromov boundary of $N$ to the Morse boundary of $G$.
	\end{enumerate}
\end{mainthm1}
\begin{proof}
	(1) follows from Theorem \ref{theorem-main} $+$ Lemma \ref{lemma-normal-intersect-freeabelian}.
	
	(2) By \cite[Corollary 5.10]{ABR-structure-invariant}, we see that the Morse boundary of $G$ is dense in the HHG boundary of $G$. Since, by (1), there does not exist a Cannon--Thurston map from the Gromov boundary (which is the same as the Morse boundary) of $N$ to the HHG boundary of $\Gamma$, we are done.
\end{proof}

Now, we are ready to answer Question \ref{ques:ct-map-free-by-cyclic}.

\begin{corfbc}
	Let $\Gamma=\F\rtimes_{\phi}\ZZ$ be an unbranched free by cyclic group such that $\Gamma$ is not quasiisometric to a product of unbounded spaces. Let $N$ be a hyperbolic normal subgroup of $\Gamma$. If $\phi$ is not atoroidal, then we have the following:
    \begin{enumerate}
        \item The inclusion $N\hookrightarrow \Gamma$ does not extend to a continuous map from the Gromov boundary of $N$ to the HHG boundary of $\Gamma$. 
        \item There does not exist a Cannon--Thurston map from the Morse boundary of $N$ to the Morse boundary of $\Gamma$.
    \end{enumerate}
\end{corfbc}
\begin{proof}
	(1) Since $\phi$ is not atoroidal, $\Gamma$ contains a subgroup isomorphic to $\ZZ^2$. Hence, we are done by Theorem \ref{main-thm-1}(1).

    (2) This follows from Theorem \ref{main-thm-1}(2).
\end{proof}

Following direct corollary adds to the discussion regarding normal and right angled Artin subgroups of surface mapping class groups. 
\begin{cormcg}
	Let $S_{g,n}$ denote the connected, orientable surface of genus $g$ and $n\geq 0$ punctures such that $3g-3+n\geq 2$. For any hyperbolic normal subgroup $N$ of Mod$(S_g^n)$, the Cannon--Thurston map does not exist from the Gromov boundary of $N$ to the HHG boundary of Mod$(S_g^n)$.
\end{cormcg}
\begin{proof}
	Since Mod$(S_g^n)$ contains $\ZZ^2$ subgroups (subgroups generated by Dehn twists around disjoint simple closed curves), the corollary follows from Theorem \ref{main-thm-1}.
\end{proof}

Let $H=H_0\ast \ZZ^2$, where $H_0$ is a hyperbolic group, and consider an automorphism $\phi:H\to H$ where $\phi|_{H_0}:H_0\to H_0$ a hyperbolic automorphism, and $\phi|_{\ZZ^2}$ the identity. Let $G=H\rtimes_{\phi}\ZZ$. In \cite[Corollary 3.7]{charney-sisto-ct-map-morse}, the authors showed that there does not exist a Cannon--Thurston map from the Morse boundary of $H$ to the Morse boundary of $G$. Note that $H$ is not a hyperbolic group. We show below that their conclusion holds for any hyperbolic normal subgroup of $G$ as well.
\begin{corollary} Let $G=H\rtimes_{\phi}\ZZ$ and $H$ be as above. 
 We have the following:
	\begin{enumerate}
		\item Every Morse element of $H$ is also a Morse element of $G$.
		\item Let $N$ be a normal hyperbolic subgroup of $G$. Then, there does not exist a Cannon-Thurston map from the Morse boundary of $N$ to the Morse boundary of $G$.
		\item For any hyperbolic normal subgroup $N$ of $G$ there exists an element of $N$ which is not Morse in $G$.
	\end{enumerate}
\end{corollary}
\begin{proof}
The proof of (1) is the same as the proof of \cite[Corollary 3.7(1)]{charney-sisto-ct-map-morse}. For the proof of (2), let $N$ be a normal subgroup of $G$. By writing $G$ as an amalgamated free product, we see that $G$ is hyperbolic relative to $\ZZ^3$. Thus, $G$ is an HHG by \cite[Theorem 9.1]{BHS2}. In the HHG structure of $G$, the top-level hyperbolic space, i.e., the coned-off Cayley graph of $G$ with respect to the cosets of $\ZZ^3$ is unbounded, and hence $G$ is not quasiisometric to a product of unbounded spaces.  Then, by Theorem \ref{main-thm-1}(2), we are done. (3) is immediate as Morse element of $G$ has North-South dynamics for the HHG boundary of $G$ (given any hyperbolic normal subgroup $N$ of $G$, by Lemma \ref{lemma-normal-intersect-freeabelian}, there will be a free abelian subgroup $A$ of rank $2$ which intersects $N$. Thus, in the intersection, we have an infinite order element which is not Morse in $G$ as it fixes pointwise the boundary of the hierarchically quasi convex hull of $A$).
\end{proof}

\begin{remark}
 All the results mentioned above do not depend on the chosen HHG structures.
	
\end{remark}

\section{Relatively hierarchical boundaries of CAT(0) groups with isolated flats and Cannon--Thurston maps}

In this section, we will prove the following theorem:
\begin{mainthm2}
    Let $G$ be a group acting geometrically on a CAT(0) space $X$ with isolated flats. Then, the identity map of $X$ extends to a $G$-equivariant homeomorphism from the CAT(0) boundary of $X$ to the relatively hierarchically hyperbolic boundary of $X$.
\end{mainthm2}

Let $G$ and $X$ be as in Theorem \ref{main-thm-2}. 

\subsection{Bowditch boundary vs CAT(0) boundary of $X$}  Let $\calF$ be the collection of flats in $X$. Note that, by Theorem \ref{theorem-cat(0)-isolated-flat}, the CAT(0) boundary $\partial X$ of $X$ is an invariant of $G$. Since the flats are isometrically embedded copies of Euclidean spaces, we have a $G$-invariant collection of spheres in the CAT(0) boundary of $X$. Also, $G$ is hyperbolic relative to a collection $\calH$ of virtually free abelian groups of rank at least $2$ (Theorem \ref{theorem-cat(0)-isolated-flat}). Let $\Gamma$ be a Cayley graph of $G$, and let $\Phi:\Gamma\to X$ be an orbit map, which is a quasiisometry by the Milnor--Svarc lemma (\cite[Proposition 8.19, Chapter I.8]{BH}). For each $g\in G$ and $H\in \calH$, we denote by $Y_{gH}$ the image of $gH$ under the map $\Phi$. Under the action of $G$, there are finitely many orbits of $\calF$, and up to conjugacy in $G$, each $H\in \calH$ is the $G$-stabilizer of some $F\in\calF$ (\cite[Lemma 3.1.2]{hruska-kleiner-isolated-flat}). If $H$ is the $G$-stabilizer of $F\in\calF$, then the Hausdorff distance between $Y_{gH}$ and $F$ is uniformly bounded. Let $\hat{\Gamma}$ be the coned-off Cayley graph of $G$ with respect to $\calH$, and let $\hat{X}$ be the coned-off space of $X$ with respect to $\calF$. Then, by \cite[Proposition 3.2]{sisto-projection}, $\Phi$ induces a quasiisometry $\hat{\Phi}:\hat{\Gamma}\to \hat{X}$. The points in the boundaries of flats are referred to as the {\em peripheral limit points.} Two peripheral limit points are said to be of {\em same type} if they lie in the boundary of some $F\in\calF.$

Note that $\hat{\Gamma}$ is a fine hyperbolic graph \cite{Bow-97}. Define $\Delta_{\infty}(\hat{\Gamma}):=V_{\infty}(\hat{\Gamma})\cup \partial \hat{\Gamma}$, where $V_{\infty}(\hat{\Gamma})$ denotes the set of infinite valence vertices of $\hat{\Gamma}$ and $\partial\hat{\Gamma}$ denotes the Gromov boundary of $\hat{\Gamma}$. Then, with respect to a topology as defined in \cite{Bow-97} (see also \cite[Section 3]{tran-boundary-relation}), $\Delta_{\infty}(\hat{\Gamma})$ is a second countable compact Hausdorff space. Moreover, as shown in \cite{Bow-97}, $\Delta_{\infty}(\hat{\Gamma})$ is homeomorphic to the \emph{Bowditch boundary} of $G$ with respect to $\calH$. In this work we will not define the Bowditch boundary explicitly, as we will not need it, but use this correspondence and the homeomorphism between the Bowditch  and CAT(0) boundaries given in Theorem \ref{theorem:relation-bowditch-cat(0)-boundary} below instead.

Before describing the homeomorphism between two boundaries, we need the following definition and the subsequent Lemma \ref{lemma:existence-of-nice-pairs}.
\begin{definition}\cite[Definition 5.7]{tran-boundary-relation}
    For $\epsilon\geq 0$ and $A\geq 0$, a pair of paths $(c,\hat{c})$ in $\hat{\Gamma}$ is said be an $(\epsilon,A)-${\em nice pair} if $c$ is an $\epsilon$-quasigeodesic ray in $\Gamma$ and $\hat{c}$ is a geodesic ray in $\hat{\Gamma}$, and each $G$-vertex of $\hat{c}$ lies in a $A$-neighborhood of $c$ with respect to the metric of $\Gamma$. For $r\geq 0$ and a geodesic ray $\alpha$ in $X$, a pair of paths $(c,\hat{c})$ is said to be an $(\epsilon,A,r)-${\em nice pair of geodesic ray $\alpha$} if $(c,\hat{c})$ is an $(\epsilon,A)-$nice pair and the Hausdorff distance between $\alpha$ and $\Phi(c)$ is at most $r$. A pair of path $(c,\hat{c})$ is said to be {\em nice pair of the geodesic ray $\alpha$} if $(c,\hat{c})$ is an $(\epsilon,A,r)-$nice pair of $\alpha$ for some $\epsilon\geq 0,A\geq 0,$ and $r\geq 0$.
\end{definition}
The following lemma shows the existence of a nice pair of a geodesic ray in $X$.
\begin{lemma}\label{lemma:existence-of-nice-pairs}
    \cite[Lemma 5.9]{tran-boundary-relation} There are positive constants $\epsilon,A,r$ such that the following holds. Let $\alpha$ be a geodesic ray in $X$ such that $\alpha \not\subset N_K(gH)$ for any $gH$ and any $K\geq 0$ (equivalently $[\alpha]\notin \partial F$ for any $F\in\calF$) there is an
$(\epsilon,A,r)-$nice pair $(c,\hat{c})$ of $\alpha$. Moreover, if $\alpha$ is a geodesic ray with initial point $\Phi(g)$, where $g\in G$, then $c$ and $\hat{c}$ could be chosen with initial point $g$.
\end{lemma}

Now, we are ready to describe the homeomorphism $f\co\partial X\to \Delta_{\infty}(\hat{\Gamma})$   from \cite[Section 6]{tran-boundary-relation}. Let $\alpha$ be a geodesic ray in $X$. If $\alpha\subset N_K(Y_{gH})$ for some $K\geq 0$ define $f([\alpha])$ to be the cone point corresponding to $gH$. Otherwise, by Lemma \ref{lemma:existence-of-nice-pairs}, there is a nice pair $(c,\hat{c})$ of the geodesic ray $\alpha$, and $f([\alpha])$ is defined as $[\hat{c}]$. Thus we have,  

\begin{theorem}\cite[Theorem 1.1]{tran-boundary-relation}\label{theorem:relation-bowditch-cat(0)-boundary}
    The Bowditch boundary of $(G,\calH)$  is $G$-equivariantly homeomorphic to the space obtained from the CAT(0) boundary $\partial X$ of $X$ by identifying the peripheral limit points of the same type.
\end{theorem}

\subsection{Relatively hierarchically hyperbolic boundary of $X$} By the relatively hierarchically hyperbolic structure given in Remark \ref{example-relhhg}, $(X,\gS)$ is a relatively hierarchically hyperbolic space. Now, we define the boundary of $(X,\gS)$, and it will be referred to as {\em relatively hierarchically hyperbolic boundary} of $X$, denoted by $\partial_{RHH}X$. 

In the definition of the hierarchically hyperbolic boundary of an HHS, we use the Gromov boundary of the domain spaces. To define the \emph{relatively} hierarchically hyperbolic boundary of $X$, we use the CAT(0) boundary of $F\in\calF$ and the Gromov boundary of $\hat{X}$. We use geodesic rays in $F$ and uniform quasigeodesic rays in $\hat{X}$ to define the boundary projections (Definition \ref{defn:boudnary-projection}). The boundary points in $\partial_{RHH}X$ and the topology on $\overline{X}_{RHH}:=X\cup\partial_{RHH}X$ are defined similarly as for the hierarchically hyperbolic boundary of an HHS. Since, in $\gS$, the orthogonality relation is empty, as a set, $\partial_{RHH}X$ has the following form:
$$\partial_{RHH}X:=(\bigsqcup_{F\in\calF}\partial F)\sqcup \partial\hat{X},$$ where $\partial F$ and $\partial\hat{X}$ denote the CAT(0) and Gromov boundaries of $F$ and $\hat{X}$, respectively.

\begin{lemma}\label{lemma-rel-bdry-haus}
    We have the following.

    \begin{enumerate}
        \item $\overline{X}_{RHH}$ is Hausdorff.
        \item For each $p\in\partial_{RHH}X$, the set $B_{{U_S},\epsilon}(p)$ forms a neighborhood basis for the topology on $\overline{X}_{RHH}$ at $p$.
    \end{enumerate}
\end{lemma}
\begin{proof}
    The proof of (1) is similar to the proof of \cite[Lemma 2.16]{DHSbound}. For HHSs, (2) is proved in \cite[Proposition 1.5]{hagen-metrizable} relying on the proof of Lemma 1.2 of the same work. In our setup, the proof of the first part of Lemma 1.2 is similar. However, for each $F\in \calF$, the corresponding domain space is $F$ itself, which is not Gromov hyperbolic. Thus, in our setting, the second part of Lemma 1.2 requires a proof. However, using Lemma \ref{lemma-cat(0)-nbhd} and following the proof of \cite[Lemma 1.2]{hagen-metrizable}, the second part of Lemma 1.2 in our setup goes through as well.
\end{proof}

For a closed convex subset $F$ of a proper CAT(0) space $X$, there is a well-defined projection map $\pi_F:X\to F$ that sends each $x\in X$ to the closest point in $F$ (\cite[Chapter II.2]{BH}).

\begin{lemma}\label{lemma-BGI}
    \cite[Lemma 1.15]{sisto-projection} There exist constants $L\geq 0$ and $R\geq 0$ such that if, for $x,y\in X$, $d(\pi_F(x),\pi_F(y))\geq L$ then any geodesic joining $x$ and $y$ intersects $B_R(\pi_F(x))$ and $B_R(\pi_F(y))$.
\end{lemma}

\begin{lemma}\label{lemma-main-technical-lemma}
Let $F$ be a flat in $X$ and let $\{\xi_n\}\subset \partial X\setminus\{\Lambda (F)\}$ be a sequence converging to a point $\xi\in \Lambda (F)$. Let $\alpha_n$ be the geodesic ray that starts from $x_0$ and represents $\xi_n$. There exists a subsequence $\{n_k\}$ of $\{n\}$ such that $t_{n_k}\to\infty$ as $k\to \infty$ and the sequence $\{\pi_F(\alpha_{n_k}(t_{n_k}))\}$ converges to $\xi$. 
\end{lemma}
\begin{proof}
    Let $\alpha$ be the geodesic ray in $X$ such that $\alpha(0)=x_0$ and $\alpha(\infty)=\xi$. Since $\xi\in\Lambda (F)$, there exists $S\geq 0$ such that $\alpha\subset N_S(F).$ As $\xi_n\to \xi$, for all $r\geq 0$ there exists $N(r)\in\mathbb{N}$ such that, for $n\geq N(r)$, $d(\alpha_n(t),\alpha(t))\leq 1$ for all $0\leq t\leq r$. This implies that, for all $r\geq 0$, there exists $N(r)\in\mathbb N$ such that $d(\alpha_n(t),\pi_F(\alpha_n(t)))\leq S+1$ for $0\leq t\leq r$ and $n\geq N(r)$. We choose $r$  sufficiently large compared to $S$ and $L$ so that $\diam_X(\pi_F(\alpha_n))\geq L$ for all $n\geq N(r)$ where $L$ is the constant from Lemma \ref{lemma-BGI}.  
    In fact, by the same argument, we see that $\{\diam_X(\pi_F(\alpha_n))\}_n$ is unbounded. Thus, by Lemma \ref{lemma-BGI} and the convexity of $F$, there exists a subsequence $\{n_k\}$ of $\{n\}$ such that $I_{n_k}=[u_{n_k},v_{n_k}]$ is a maximal connected subinterval of $[0,\infty)$ satisfying $\alpha_{n_k}([u_{n_k},v_{n_k}])\subset N_R(F)$. Moreover, $d(\pi_F(t),\pi_F(u_{n_k}))\leq L$ whenever $t\leq u_{n_k}$ and  $d(\pi_F(t),\pi_F(v_{n_k}))\leq L$ whenever $t\geq v_{n_k}$. Note that, if $v_{n_k}=\infty$, then $\alpha_{n_k}([u_{n_k},\infty))\subset N_R(F)$, and this implies that $\xi_{n_k}\in\Lambda (F)$, a contradiction to our hypothesis that $\xi_n\subset \partial X\setminus\{\Lambda (F)\}$ for all $n$. Thus, for all $n$, $v_n<\infty$. Since $\{\diam_X(\pi_F(\alpha_n))\}_n$ is unbounded, the subintervals $I_{n_k}$ are unbounded. Now, choose $t_{n_k}=v_{n_k}$.  Since $\pi_F$ moves every point on $\alpha_{n_k}(I_{n_k})$ a distance of at most $R$, $\{\pi_F(\alpha_{n_k}(t_{n_k}))\}$ is an unbounded sequence in $F$. As $d(\alpha_{n_k}(t_{n_k}),\pi_F(t_{n_k}))\leq R$, $\{\pi_F(\alpha_{n_k}(t_{n_k}))\}$ converges to $\xi$ by Lemma \ref{lemma:weakly-visible}.
\end{proof}
\begin{proof}[Proof of Theorem \ref{main-thm-2}]

We are now ready to define a map $\psi:X\cup\partial X\to X\cup\partial_{RHH}X$ as follows. The restriction of $\psi$ to $X$ is the identity map of $X$ to itself. Fix a base point $x_0\in X$. From now on, all geodesic rays in $X$ start at $x_0$. Let $\alpha$ be a geodesic ray in $X$. Since $X$ has isolated flat property with respect to $\calF$, either there exists a unique $F\in\calF$ such that $\alpha\subset N_K(F)$ (this is same as saying that there exists $K'\geq 0$ such that $\alpha\subset N_{K'}(Y_{gH})$ for a unique coset of some $H\in\calH$) for some $K\geq 0$ or, by Lemma \ref{lemma:existence-of-nice-pairs}, there exists a nice pair $(c,\hat{c})$ of the geodesic ray $\alpha$. If $\alpha\subset N_K(F)$, then $[\alpha]$ represents a limit point of $F$, i.e. $[\alpha]\in\Lambda  (F)$. Define $\psi([\alpha])=[\alpha_F]$, where $\alpha_F$ is the geodesic ray in $F$ with base point $\pi_F(x_0)$ and $\alpha_F(\infty)=\alpha(\infty)$. Otherwise, define $\psi([\alpha])=[\hat{\Phi}(\hat{c})]$, where $(c,\hat{c})$ is a nice pair of $\alpha$  and $\hat{\Phi}$ is the induced quasiisometry from $\hat{\Gamma}$ to $\hat{X}$. It is clear that $\psi$ is $G$-equivariant.

Now we will prove that $\psi$ is well-defined. Let $\alpha_1$ and $\alpha_2$ are two geodesic rays such that $[\alpha_1]=[\alpha_2]$. Suppose $\alpha_1\subset N_K(F)$ for some $K\geq0$ and $F\in\calF$. Then, $\alpha_2\subset N_{K'}(F)$ for some $K'\geq 0$. Thus, $\alpha_1$ and $\alpha_2$ represent the same point of $\Lambda(F)$, and therefore $\psi([\alpha_1])=\psi([\alpha_2])$. Suppose $[\alpha_1]\notin\Lambda (F)$ for any $F\in\calF$. Then, the same is true for $[\alpha_2]$. Let $(c_1,\hat{c_1})$ and $(c_2,\hat{c_2})$ be nice pairs of the rays $\alpha_1$ and $\alpha_2$, respectively. By definition of a nice pair, Hausdorff distances $Hd_X(\Phi(c_1),\alpha_1)$ and $Hd_X(\Phi(c_2),\alpha_2)$ are finite. Since $[\alpha_1]=[\alpha_2]$, we see that $Hd_X(\Phi(c_1),\Phi(c_2))$ is finite. This implies that $Hd_{\Gamma}(c_1,c_2)$ is finite, and hence $Hd_{\hat{\Gamma}}(c_1,c_2)$ is finite. By \cite[Remark 5.8]{tran-boundary-relation}, $Hd_{\hat{\Gamma}}(c_1,\hat{c_1})<\infty$ and $Hd_{\hat{\Gamma}}(c_2,\hat{c_2})<\infty$. Thus, $\hat{c_1}$ and $\hat{c_2}$ represent the same point of $\partial \hat{\Gamma}$. Hence, $\hat{\Phi}(\hat{c_1})$ and $\hat{\Phi}(\hat{c_2})$ represent the same point of $\hat{X}$. Thus, $\psi$ is well defined.

Lemma \ref{theorem-RHHG-boundary-sameas-CAT(0)-bdry} below finishes the proof of the theorem. 
\end{proof}
\begin{lemma}\label{theorem-RHHG-boundary-sameas-CAT(0)-bdry}
   The map $\psi$ is a homeomorphism.
\end{lemma}
\begin{proof}
Since $X\cup\partial X$ is compact and, by Lemma \ref{lemma-rel-bdry-haus}(1), $\overline{X}_{RHH}$ is Hausdorff, to prove that $\psi$ is a homeomorphism, it is sufficient to prove that $\psi$ is bijective and continuous.

\noindent {\em Bijectivity of  $\psi$:} The restriction of $\psi$ to \[\{[\alpha]:\alpha \text{ is a geodesic ray representing a point of } \Lambda(F) \text{ for some } F\in \calF \}\]
is clearly a bijection onto $\bigsqcup_{F\in\calF}\partial F$. So let $\hat{\beta}$ be a quasigeodesic ray in $\hat{X}$. Then, $\hat{\Phi}^{-1}(\hat{\beta})$ is a quasigeodesic ray in $\hat{\Gamma}$, where $\hat{\Phi}^{-1}$ denotes the quasi-inverse of $\hat{\Phi}$. By \cite[Lemma 5.10]{tran-boundary-relation}, there exists a geodesic ray $\beta$ in $X$ such that $\beta(\infty)\notin \Lambda(F)$ for any $F\in\calF$, and if $(c,\hat{c})$ is an arbitrary nice pair of $\beta$ then, $Hd_{\hat{\Gamma}}(\hat{\Phi}^{-1}(\hat{\beta}),\hat{c})<\infty$. This implies that $Hd_{\hat{X}}(\hat{\beta},\hat{\Phi}(\hat{c}))<\infty$, and therefore $\psi$ is surjective. 

Let $\alpha$ and $\alpha'$ be two geodesic rays in $X$, and let $(c,\hat{c})$ and $(c',\hat{c}')$ be nice pairs of $\alpha$ and $\alpha'$, respectively, such that $[\hat{\Phi}(\hat{c})]=[\hat{\Phi}(\hat{c}')]$. This implies that $[\hat{c}]=[\hat{c}']$ in $\hat{\Gamma}$. Then, by \cite[Lemma 5.11]{tran-boundary-relation}, $[\alpha]=[\alpha'].$ Thus, $\psi$ is bijective.

\noindent{\em Continuity of $\psi$.} 
{\bf Case 1.} We first prove the continuity of $\psi$ on the points which do not lie on the limit sets of any of the flats: let $\alpha$ be a geodesic ray in $X$ such that $[\alpha]\notin \Lambda(F)$ for any $F\in \calF$ and let $(c,\hat{c})$ be a nice pair for $\alpha$. Let $\{\xi_n\}$ be a sequence in $\partial X$ such that $\xi_n\to [\alpha]$ as $n\to\infty$. Let $\alpha_n$ be the geodesic ray representing $\xi_n$. There are three subcases to be considered.

{\bf Subcase 1.} Suppose, except for finitely many $n$, $[\alpha_n]\notin \Lambda(F)$ for any $F\in\calF$. By Lemma \ref{lemma:existence-of-nice-pairs}, there exists nice pairs $(c_n,\hat{c}_n)$ for each $n$. Since, by Theorem \ref{theorem:relation-bowditch-cat(0)-boundary}, $\partial X$ is homeomorphic to $\Delta_{\infty}(\hat{\Gamma})$, $f(\xi_n)\to f([\alpha])$ as $n\to\infty$ where $f$ is the homeomorphism given by this theorem. By definition of $f$, this implies that $[\hat{c}_n]\to [\hat{c}]$ as $n\to\infty$. Since $\hat{\Phi}$ is a quasiisometry, $\hat{\Phi}([\hat{c_n}])\to \hat{\Phi}([\hat{c}])$ as $n\to\infty$. Finally, by definition of $\psi$, it follows that $\psi(\xi_n)\to \psi([\alpha])$ as $n\to \infty$.

{\bf Subcase 2.} Suppose, except for finitely many $n$, there exists $F_n\in \calF$ such that $\xi_n\in \Lambda(F_n)$. Observe that, as $[\alpha]\notin \Lambda(F)$ for any $F\in\calF$, there does not exist $F\in\calF$ such that $\xi_n\in \Lambda(F)$ for infinitely many $n$. Note that the points of each $\partial F$ are remote points to $[\hat{\Phi}(\hat{c})]$ in $\partial_{RHH}X$. It follows from the definition of topology on $\partial_{RHH}X$ that $\psi(\xi_n)\to\psi([\alpha])$ if and only if $\{v_n\}\to\psi([\alpha])$ as $n\to \infty$, where $v_n$ is the cone point corresponding to $F_n\in\calF$. As discussed above, for each $n$, there exists unique $g_nH_i$, for some $g_n\in G$ and $H_i\in\calH$ such that $\alpha_n\subset N_R(Y_{g_nH_i})$ for some $R\geq 0$. Let $w_n$ be the cone point corresponding to $g_nH_i$ in $\hat{\Gamma}$. Since the map $f$ is continuous, $f([\alpha_n])=w_n\to [\hat{c}]$ as $n\to\infty$ in $\Delta_{\infty}(\hat{\Gamma})$. By \cite[Proposition 8.5]{Bow-97}, the subspace topology on $\partial\hat{\Gamma}$ from $\Delta_{\infty}(\hat{\Gamma})$ coincides with the usual topology on $\partial\hat{\Gamma}$. Thus, $w_n\to[\hat{c}]$ as $n\to\infty$ in the usual topology on $\partial\hat{\Gamma}$. Hence, $v_n=\hat{\Phi}(w_n)\to \psi([\alpha])$ as $n\to\infty$.

{\bf Subcase 3.} Suppose that neither Subcase 1 nor Subcase 2 holds. Then, given a subsequence $\{[\alpha_{n_k}]\}$ of $\{[\alpha_n]\}$, there exists a subsequence $\{[\alpha_{n_{k_{l}}}]\}$ of $\{[\alpha_{n_k}]\}$ such that for $\{[\alpha_{n_{k_{l}}}]\}$ either Subcase 1 or Subcase 2 holds. Hence, $\psi(\xi_n)\to\psi([\alpha])$ by Lemma \ref{lemma-subseq-convergence}.

{\bf Case 2.} Let $\alpha$ be the geodesic ray that represents a point $\xi\in\Lambda(F)$ for some $F\in\calF$. Let $\{\xi_n\}$ be a sequence in $\partial X$ such that $\xi_n\to[\alpha]$ as $n\to\infty$. Let $\alpha_n$ be the geodesic ray that starts from $x_0$ and represents $\xi_n$. The following subcases are to be considered.

{\bf Subcase 1.} Suppose, except for finitely many $n$, $\xi_n\in\Lambda(F)$. Since changing the base point is a homeomorphism of $\partial X$ and $F$ is convex, $\psi(\xi_n)\to [\alpha]$ as $n\to\infty$.

{\bf Subcase 2.} Suppose, except for finitely many $n$, $\xi_n\notin\Lambda(F)$. By Lemma \ref{lemma-main-technical-lemma}, there exists a subsequence $\{\alpha_{n_k}\}$ of $\{\alpha_n\}$ such that $\pi_F(\alpha_{n_k}(t_{n_k}))\to\xi:=[\alpha]$ as $k\to\infty$. By abuse of notation, we still denote the subsequence $\{\alpha_{n_k}\}$ by $\{\alpha_n\}$. Note that, for all $n$, $\psi([\alpha_n])$ are remote points to $\xi$ in $\partial_{RHH}X$. To show that $\psi(\xi_n)\to\psi(\xi)$, we show that the boundary projection of $\psi(\xi_n)$ converges to $\psi(\xi)$. For each $n$, there are two possibilities. Suppose $\xi_n\in\Lambda(F')$ for some $F\in\calF$. Then, $\alpha_n\subset N_K(F')$ for some $K\geq 0$. Choose $t\geq t_n$ such that $\alpha_n(t)\notin N_E(\pi_{F'}(F))$ and $f'\in F'$ such that $f'\notin N_E(\pi_{F'}(F))$, where $E$ is the hierarchical constant. Then, by the consistency axiom, we see that $d(\pi_F(f'),\pi_F(F'))\leq E$ and $d(\pi_F(\alpha_n(t),\pi_F(F'))\leq E$. Using the triangle inequality, we have that $d(\pi_F(f'),\pi_F(\alpha_n(t)))\leq 2E$. On the other hand, by Lemma \ref{lemma-main-technical-lemma}, we have that $d(\pi_F(\alpha_n(t),\pi_F(\alpha_n(t_n)))\leq L$. This implies that $d(\pi_F(f'),\pi_F(\alpha_n(t_n))\leq 2E+L$ and thus $d(\pi_F(F'),\pi_F(\alpha_n(t_n)))\leq 3E+L.$ Suppose $\xi_n\notin \Lambda(F)$ for any $F\in\calF$. Let $(c_n,\hat{c}_n)$ be a nice pair corresponding to $\alpha_n$. Since $Hd_{\hat{X}}(\alpha_n,\hat{\Phi}(\hat{c}_n))<\infty$, choose a point $y_n\in\hat{\Phi}(\hat{c}_n)$ and $z_n\in\alpha_n$ such that geodesic joining $y_n$ and $z_n$ in $\hat{X}$ does not intersect $E$-neighborhood of the cone-point corresponding to $F$ and $z_n$ is also far away from $\alpha_n(t_n)$. Then, by the bounded geodesic image axiom, $d(\pi_F(y_n),\pi_F(z_n))\leq E$. Since $d(\pi_F(z_n),\pi_F(\alpha_n(t_n)))\leq L$, $d(\pi_F(y_n),\pi_F(\alpha_n(t_n)))\leq L+E$. Thus, the boundary projection of $\psi([\alpha_n])$ is uniformly close to $\pi_F(\alpha_n(t_n))$. As $\{\pi_F(\alpha_n(t_n))\}$ converges to $\xi$, using Lemma \ref{lemma-subseq-convergence}, $\psi(\xi_n)\to\psi(\xi)$.

{\bf Subcase 3.} Suppose that neither Subcase 1 nor Subcase 2 holds. Then, given a subsequence $\{[\alpha_{n_k}]\}$ of $\{[\alpha_n]\}$, there exists a subsequence $\{[\alpha_{n_{k_{l}}}]\}$ of $\{[\alpha_{n_k}]\}$ such that for $\{[\alpha_{n_{k_{l}}}]\}$ either Subcase 1 or Subcase 2 holds. Hence, $\psi(\xi_n)\to\psi([\alpha])$ by Lemma \ref{lemma-subseq-convergence}.

{\bf Case 3.} Suppose $\{x_n\}$ is a sequence in $X$ such that $x_n\to\xi$ for some $\xi\in\partial X$. Suppose first that $\xi\in \Lambda(F)$ for some $F\in\calF$. If except for finitely many $n$, $x_n\in N_K(F)$ for some $K\geq 0$. Then, it is clear that $\pi_F(x_n)\to\xi$ as $n\to\infty$, and therefore $x_n\to\psi(\xi)$ as $n\to\infty$ in $X\cup\partial_{RHH}X$. If, except for finitely many $n$, $x_n$ does not belong to any neighborhood of $F$, then, from the proof of Lemma \ref{lemma-main-technical-lemma}, it follows that $\pi_F(x_n)\to \xi$ as $n\to \infty$ and thus we are done. Note that given a subsequence of $\{x_n\}$, there exists a subsequence $\{x_{n_k}\}$ of $\{x_n\}$ such that either of the previous two situations holds. Hence, $x_{n_k}\to\psi(\xi)$ as $k\to\infty$. Now, by invoking Lemma \ref{lemma-subseq-convergence}, we are done. Suppose $x_n\to\xi$ and $\xi\notin \Lambda(F)$ for any $F\in\calF$. Let $[x_0,x_n]_X$ and $[x_0,x_n]_{\hat{X}}$ be geodesic joining $x_0$ and $x_n$ in $X$ and $\hat{X}$, respectively. Since $\hat{\Gamma}$ and $\hat{X}$ are quasiisometric, using \cite[Theorem 1.12(3)]{drutu-sapir-tree-graded}, the Hausdorff distance between $[x_0,x_n]_{X}$ and $[x_0,x_n]_{\hat{X}}$ is uniformly bounded. As $x_n\to\xi$, the geodesics $[x_0,x_n]_X$ fellow travel $\alpha$ for longer and longer time, where $\alpha$ is the geodesic ray that represent $\xi$. This implies that geodesics $[x_0,x_n]_{\hat{X}}$ fellow travel $\hat{\Phi}(\hat{c})$ for longer and longer time, where $(c,\hat{c})$ is a nice pair for $\alpha$. Thus, $x_n\to\psi(\xi)$ in $X\cup\partial_{RHH}X$.
\end{proof}


\begin{corollary}\label{cor-density-of-top-level-graph}
    Any $G$-orbit of an element of $\partial_{RHH}X$ is dense in $\partial_{RHH}X$. In particular, the Gromov boundary $\partial \hat{X}$ of $\hat{X}$ is dense in $\partial_{RHH}X$.
\end{corollary}
\begin{proof}
     Since $G$ contains a rank-one isometry of $X$ (\cite[Lemma 2.30]{BGGGS}), by \cite[Lemma 5.2]{ursula-rank-one-isom}, any $G$-orbit of an element of $\partial X$ is dense in $\partial X$. By Theorem \ref{main-thm-2}, we know that $\partial X$ is $G$-equivariantly homeomorphic to $\partial_{RHH}X$. Hence, the lemma.
\end{proof}
An immediate consequence of Corollary \ref{cor-density-of-top-level-graph} is the following:
\begin{corollary}
    For any infinite normal subgroup of $G$, $\Lambda(N)=\partial_{RHH}X$.
\end{corollary}
\begin{proof}
    Since, by Corollary \ref{cor-density-of-top-level-graph}, $\partial\hat{X}$ is dense in $\partial_{RHH}X$, the proof follows from the proof of Lemma \ref{lemma:limit-set-of-normal-subgroup}(2).
\end{proof}
Next, we discuss a few applications of Theorem \ref{theorem-RHHG-boundary-sameas-CAT(0)-bdry}. Before that, we define the CT map for relative HHG boundaries of CAT(0) groups with isolated flats.

\begin{definition}\label{defn:CT-maps-isolated-flats}
    Let $H$ be a subgroup of a group $G$. Suppose $H$ and $G$ act geometrically on $Y$ and $X$, respectively, where $X$ and $Y$ are either hyperbolic spaces or CAT(0) spaces with isolated flats. 
    Let $y\in Y$ and $x\in X$. 
    A {\em Cannon--Thurston (CT)} map for relative HHG boundaries for the pair $(H,G)$ is the map $ f:\partial_{RHH} Y\to \partial_{RHH} X$ defined by the following:
    $$f(\lim_{i\to\infty}h_iy)=\lim_{i\to\infty}h_ix$$ provided $f$ is continuous and $f(\xi)$ does not depend on the sequence $\{h_i\}$ such that $h_iy\to \xi$ as $i\to\infty.$
\end{definition}
\begin{remark}
    A CT map for CAT(0) boundaries of $H$ and $G$ is defined similarly by replacing relative HHG boundaries of $H$ and $G$ by CAT(0) boundaries of $H$ and $G$, respectively, in Definition \ref{defn:CT-maps-isolated-flats}.
\end{remark}
The following two propositions recover some of the results of \cite{BGGGS}.
\begin{proposition}
     Let $G$ be a CAT(0) group with isolated flats. 
    For an infinite hyperbolic normal subgroup $N$ of $G$, the pair $(N,G)$ does not admit a Cannon--Thurston map for relative HHG boundaries. Moreover, there does not exist a Cannon--Thurston map from the Gromov boundary of $N$ to the CAT(0) boundary of $G$.
\end{proposition}
\begin{proof}
    The ``moreover'' statement, which is given by \cite[Theorem 5.3]{BGGGS}, is clear from Theorem \ref{main-thm-2}. To prove the first statement, we verify the hypotheses of Lemma \ref{lemma-nonexistence-ct}. Since $G$ contains $\mathbb Z^2$, there exists a subgroup $K$ of $G$ isomorphic to $\mathbb Z^2$ such that $N\cap K\neq\phi$ by Lemma \ref{lemma-normal-intersect-freeabelian}. Note that $K\subset N_r(gH)$ for some $g\in G$ and $H\in\mathcal H$, and thus the elements of $K$ fix pointwise the limit set of the corresponding flat in $X$. Since $N$ is normal, $\Lambda (N)=\partial_{RHH}X$. Now, the element $h\in N\cap K$ satisfies both the hypotheses of Lemma \ref{lemma-nonexistence-ct}. Hence, we are done.
\end{proof}
\begin{proposition}
    Let $G$ be a CAT(0) group with isolated flats and let $H$ be a normal subgroup of $G$ which is also CAT(0) with isolated flats. Then, the pair $(H,G)$ does not admit a Cannon--Thurston map for relative HHG boundaries.
\end{proposition}
\begin{proof}
    By Theorem \ref{main-thm-2}, the relative HHG boundaries of $H$ and $G$ are homeomorphic to their CAT(0) boundaries. Hence, the result follows from \cite[Theorem 5.6]{BGGGS}.
\end{proof}

\bibliographystyle{alpha}
\def\bibfont{\footnotesize}
\bibliography{main}
\end{document}